\documentclass[11pt]{amsart}
\usepackage{mypackage}
\usepackage{myformat}

\newcommand {\Neg}{\mathscr N}
\newcommand {\La}{\mathscr L}

\newcommand {\Meas}{\mathscr M}    

\begin{document} 
\title[Kelley's Theorem]
{Kelley's Theorem and Some Related Results}
\mybic
\date \today 
\subjclass[2010]{
Primary: , 
Secondary: 
.} 

\begin{abstract} 
We give an elementary proof of Kelley's theorem
based on a minimax argument. Some applications
to related problems are also developed.
\end{abstract}

\maketitle

\section{Introduction and notation.}
The well known problem launched decades ago by
Dorothy Maharam \cite{maharam} of whether a 
Boolean algebra admits a strictly positive, additive 
set functions defined thereon -- the so called 
Maharam problem -- has motivated a long lasting 
stream of mathematical research, see \cite{jech} 
for a comprehensive review. In much of this literature 
the focus has been on complete (or $\sigma$ complete) 
Boolean algebras and countably additive set functions. 
One of the first papers on this topic was that of Kelley
\cite{kelley} and it was also one of the few treating 
the case of finitely additive set functions. His approach 
in terms of intersection numbers is still one of the few 
results characterizing the situation of a finitely additive, 
strictly positive set function. Another one was obtained 
much later by Jech et al \cite{balcar_jech_pazak}.
 
In this paper we present a very simple proof of this 
important result based on the minimax theorem. 
Another simple proof was obtained in recent years 
by Aversa and Bhaskara Rao \cite{aversa_rao} using 
linear programming (see other references quoted therein).
We also develop a number of implications that justify
interest for the method proposed here.

In this paper terms such as measure or probability
will always refer to finitely additive set functions.

\section
{Kelley's Theorem.}

Let $\A$ be an algebra of subsets of some non empty 
set $\Omega$ and $\A_+=\A\setminus\{\emp\}$.
$\Prob(\A)$ designates  the family of (finitely additive) 
probabilities defined on $\A$ and, for each $m\in\Prob(\A)$, 
let $m(f)$ indicate the integral of $f$ with respect to 
$m$ -- if well defined. $m\in\Prob(\A)$ is strictly positive 
if $m(A)>0$ for all $A\in\A_+$.

For given $\Bor\subset\A_+$ write the set of finite 
sequences from $\Bor$ as $\Seq[0]\Bor$. With each 
$\beta\in\Seq[0]\Bor$%
\footnote{
The elements of $\beta\in\Seq[0]\Bor$ need of
course not be distinct.
} 
we can associate the following function on $\Omega$:
\begin{equation}
s(\beta)
	=
\frac{1}{\abs\beta}\sum_{B\in\beta}\set B
\end{equation}
where $\abs\beta$ designs the length of the sequence 
$\beta$. Kelley \cite{kelley} defined the intersection number 
of $\Bor$ as 
\begin{equation}
\label{I}
I(\Bor)
	=
\inf_{\beta\in\Seq[0]\Bor}
\sup_\omega
s(\beta)(\omega)
\end{equation}
Clearly, $0\le I(\Bor)\le1$; if $\Bor$ contains an infinite, 
disjoint collection of sets then necessarily $I(\Bor)=0$. 
If $\Bor\subset\A_+$, we introduce the family $\Sim(\Bor)$ 
of convex combinations of indicators of sets in $\Bor$. 
Clearly, 
$\{s(\beta):\beta\in\Seq[0]\Bor\}\subset\Sim(\Bor)$.
The closure of a set $A$ of real valued functions on
$\Omega$ with respect to the topology of uniform
distance will be denoted by $\cl[u]A$.

\begin{theorem}[Kelley, 59]
\label{th kelley}
An algebra $\A$ of sets admits a strictly positive, finitely
additive probability measure if and only if $\A$ may be 
written in the form
\begin{equation}
\label{kelley}
\A
	=
\{\emp\}
\cup
\bigcup_n\Bor_n
\qtext{with}
I(\Bor_n)>0
\qtext{for}
n=1,2,\ldots.
\end{equation}
\end{theorem}

\begin{proof}
Necessity is obvious -- if $m\in\Prob(\A)$ is strictly
positive, take $\Bor_n=\{B\in\A:m(B)>1/n\}$. As for 
sufficiency, since each $f\in\Sim(\Bor)$ with values 
in $\Q$ belongs to 
$\{s(\beta):\beta\in\Seq[0]\Bor\}$,
then
$
\Sim(\Bor)
=
\cl[u]{\big\{s(\beta):\beta\in\Seq[0]\Bor\big\}}
$.
Moreover, each $f\in\Sim(\Bor)$ has finite range and 
every $A\in\A_+$ admits some $m\in\Prob(\A)$ with 
$m(A)=1$. Then, we deduce from \cite[Corollary 3.3]{sion}
\begin{align}
\label{sion}
I(\Bor)
&=
\inf_{\beta\in\Seq[0]\Sim}\sup_\omega s(\beta)(\omega)
=
\inf_{f\in\Sim(\Bor)}\sup_\omega f(\omega)
=
\inf_{f\in\Sim(\Bor)}\sup_{m\in\Prob(\A)}m(f)
=
\sup_{m\in\Prob(\A)}\inf_{f\in\Sim(\Bor)}m(f).
\end{align}
Under \eqref{kelley} each $n\in\N$ admits 
$m_n\in\Prob(\A)$ satisfying 
$\inf_{B\in\Bor_n}m_n(B)
>
I(\Bor_n)/2$. Then, $\sum_n2^{-n}m_n\in\Prob(\A)$ is 
strictly positive.
\end{proof}

Since each $\Bor$ with $I(\Bor)>0$ can contain at 
most finitely many, pairwise disjoint sets, it follows 
from \eqref{kelley} that a family of pairwise disjoint 
sets in $\A_+$ must be countable. This is the well 
known CC ({\it countable chain}) necessary condition 
formulated by Maharam and long conjectured to be 
sufficient until Gaifman \cite{gaifman} counterexample 
of a Boolean algebra possessing the CC property but 
lacking a strictly positive measure.

Aversa and Bhaskara Rao \cite{aversa_rao} make use 
of Tychonoff Theorem to prove Kelley's Theorem. This
is also important in our proof, although indirectly, via 
Sion's lemma.

\section{Some Related Results}

The relative advantage of our proof, apart from 
simplicity, is the great ease of generalization. 
Denote by $\La(\A)$ the vector space spanned by the
indicators of sets in $\A$ and for each $\rho:\La\to\R_+$, 
let
\begin{equation}
\label{N rho}
\Neg(\rho)
	=
\{A\in\A:\rho(\set A)=0\}.
\end{equation}
A set function $m\in ba(\A)_+$ such that 
$\Neg(m)\subset\Neg(\rho)$ 
is said to be strictly $\rho$-positive. If 
$\rho(\set A)\ge m(A)$ for all $A\in\A$ then $m$ is
said to be $\rho$-dominated.

\begin{theorem}
\label{th kelley pi}
Let $\pi$ be a monotone, sublinear functional on $\La(\A)$. 
There exists a $\pi$ dominated and strictly $\pi$ positive 
$m\in ba(\A)_+$ if and only if $\A$ may be written 
in the form
\begin{equation}
\label{kelley pi}
\A
	=
\Neg(\pi)
\cup\bigcup_n\Bor_n
\qtext{with}
I_\pi(\Bor_n)
	\equiv
\inf_{\beta\in\Seq[0]\Bor}\pi\big(s(\beta)\big)
	>
0
\qquad
n=1,2,\ldots
\end{equation}
Moreover, the set function $m$ may be chosen to be 
a probability if and only if $\pi(1)\ge1\ge-\pi(-1)$.
\end{theorem}

\begin{proof}
The proof of Theorem \ref{th kelley} remains true 
after replacing $I$ with $I_\pi$ provided we can 
show that
\begin{equation}
\label{attain}
\pi(f)
	=
\sup_{m\in ba(\A,\pi)_+}\int fdm
\qquad
f\in\La(\A)
\end{equation}
and that set
\begin{equation}
ba(\A,\pi)_+
	=
\Big\{m\in ba(\A)_+:
\pi(h)\ge\int hdm\text{ for all }h\in\La(\A)\Big\}
\end{equation}
is convex and weak$^*$ compact. Both claims are, 
however, obvious: the former follows from Hahn 
Banach Theorem and the representation of linear 
functionals on $\La(\A)$ (see \cite[Chapter 3]{rao} 
and ultimately \cite{horn_tarski}); the latter from 
Tychonoff Theorem. If $m(\Omega)=1$ then 
necessarily $\pi(1)\ge1\ge-\pi(-1)$; conversely, if 
$\pi(1)\ge1\ge-\pi(-1)$ then, by well known arguments,
the functional on $\La(\A)$ defined by letting 
$\hat\pi(f)
	=
\inf_{a\in\R}\pi(a+f)-a$ is monotone, sublinear and
additive with respect to constants so that
$\hat\pi(1)=1=-\hat\pi(-1)$. Clearly, $\pi\ge\hat\pi$.
If $\hat m\in ba(\A)_+$ is $\hat\pi$-dominated it is 
then a probability.
\end{proof}

As in Theorem \ref{th kelley}, the decomposition
\eqref{kelley pi}, although necessary and sufficient,
is not very handy to use. An easier condition is obtained
by imposing a constraint on the degree of non linearity 
of $\pi$.

\begin{lemma}
\label{lemma kelley pi}
Let $\pi$ be a monotone, sublinear functional on
$\La(\A)$ satisfying the property
\begin{equation}
\label{m}
\mathfrak m(\pi)
	\equiv
\sup
\frac{\sum_{i=1}^Na_i\pi(f_i)-\pi\big(\sum_{i=1}^Na_if_i\big)}
{\pi\big(\sum_{i=1}^Na_if_i\big)}
	<
\infty
\end{equation}
the supremum being over all convex combinations of
elements of $\La(\A)_+$ such that 
$\pi\big(\sum_{i=1}^Na_if_i\big)
	>
0$.
Then there exists a strictly $\pi$-positive $m\in ba(\A)_+$
which is $\pi$-dominated.
\end{lemma}

\begin{proof}
Let $\{A_1,\ldots,A_N\}\in\Seq[0]{\Bor_n}$ with
$\Bor_n
=
\{A\in\A:\pi(\set A)>1/n\}$. Then,
$
\pi\Big(\frac1N\sum_{i=1}^N\set{A_i}\Big)
\ge
\frac1N\pi(\set{A_1})
>0
$
and, by the assumption,
\begin{align*}
\pi\Big(\frac1N\sum_{i=1}^N\set{A_i}\Big)
\ge
\frac{1}{1+\mathfrak m(\pi)}\frac1N\sum_{i=1}^N\pi(\set{A_i})
\ge
\frac{1/n}{1+\mathfrak m(\pi)}.
\end{align*}
Thus, the decomposition \eqref{kelley pi} holds.
\end{proof}

Considering the role played in \eqref{kelley pi} by the
collection $\Neg(\pi)$, one may invert the perspective
adopted in Theorem \ref{th kelley pi} and raise the 
question whether a pre assigned family of sets 
$\Neg\subset\A$ coincides with the collection of 
null sets of some $m\in\Prob(\A)$, i.e. with the set
$\Neg(m)=\{A\in\A:m(A)=0\}$. Define
$\Prob(\A,\Neg)
	=
\{m\in\Prob(\A):\Neg\subset\Neg(m)\}$.
We can modify definition 
\eqref{I} into the following: 
\begin{equation}
\label{IN}
I_\Neg(\Bor)
	=
\inf_{\beta\in\Seq[0]\Bor}
\inf_{N\in\Neg}
\sup_{\omega\in N^c}s(\beta)(\omega)
\qquad
\Bor\subset\A_+.
\end{equation}

The proof of the following Corollary may be given in 
terms of quotient algebras, as clearly remarked by Gaifman
\cite[p. 61]{gaifman}, but ours is much simpler%
\footnote{
A different proof of the following Corollary appears in
\cite[Corollary 5]{Projection_2019}.
}. 
An ideal of sets is of course a collection closed with
respect to union and to subsets.

\begin{corollary}
\label{cor kelley}
Let $\Neg\subset\A$. Then, $\Neg=\Neg(m)$ for some 
$m\in\Prob(\A)$ if and only if $\Neg$ is a proper ideal 
(of sets) and if $\A$ admits the representation 
\begin{equation}
\label{kelley N}
\A
	=
\Neg\cup\bigcup_n\Bor_n
\qtext{with}
I_\Neg(\Bor_n)>0
\qtext{for}
n=1,2,\ldots.
\end{equation}
\end{corollary}

\begin{proof}

The functional defined on $\La(\A)$ by letting 
$\pi_\Neg(f)
=
\inf_{N\in\Neg}\sup_{\omega\in N^c}f(\omega)$
is monotone and positively homogeneous by 
definition and subadditive because $\Neg$ is
an ideal. Moreover, given that $\Neg$ is proper,
$\pi_\Neg(1)\ge1\ge-\pi_\Neg(-1)$. The claim
follows from Theorem \ref{th kelley pi}.
\end{proof}

If $\pi$ is as in Theorem \ref{th kelley pi} and 
$\pi(1)>0$, then $\Neg(\pi)$ is a proper ideal 
and Corollary \ref{cor kelley} may be used to 
determine the existence of a strictly $\pi$ 
positive probability, not necessarily $\pi$-dominated.

Fix $\Meas\subset\Prob(\A)$. By choosing
$\Neg=\bigcap_{m\in\Meas}\Neg(m)$,
Corollary \ref{cor kelley} provides an answer to the
question of whether a given subfamily of $\Prob(\A)$ 
is weakly dominated. The notion of weak domination 
appears in \cite[p. 159]{rao} under the name of weak 
absolute continuity. The corresponding question of 
whether a given set $\Meas$ is dominated -- i.e. each
$m\in\Meas$ is absolutely continuous with respect
to a fixed $m_0$ -- has recently been characterized in 
\cite{JMAA_2019}.
On examining the proof of Theorem \ref{th kelley pi},
the only properties of $\Prob(\A,\Neg)$ that are used 
are convexity, weak$^*$ compactness and \eqref{attain}
which translates into
\begin{equation}
\label{norming}
m(f)
	=
\inf_{N\in\Neg}\sup_{\omega\in N^c}f(\omega).
\end{equation} 
In case $\Neg$ is the ideal of null sets of a given family 
$\Meas\subset\Prob(\A)$ these same properties are 
also true of the set
\begin{equation}
\Meas^*
	=
\cco[*]
{\{m_A:m\in\Meas,\ A\in\A,\ m(A)>0\}}
\end{equation}
where $m_A\in\Prob(\A)$ is defined as the restriction%
\footnote{
That is $m_A(B)=m(A\cap B)/m(A)$ for each $B\in\A$.
} 
of $m$ to $A$
and $\cco[*]{}$ denotes the weak$^*$-closed convex 
hull. Thus the probability that weakly dominates $\Meas$, 
if it exists, can be taken to be an element of $\Meas^*$. 
This remark delivers a version of a well known result
of Halmos and Savage \cite[Lemma 7]{halmos_savage}:

\begin{corollary}[Halmos and Savage, 49]
\label{cor hs}
If $\Meas\subset\Prob(\A)$ is weakly dominated it then
admits a weakly dominating subset which is countable.
\end{corollary}

\section{A.s. Rankings}
Following the intuitions of de Finetti \cite{definetti},
probability should be deduced endogenously from
some decision problem. In this final section we 
investigate whether an {\it a priori} given partial 
order $\ge_*$ defined%
\footnote{
The symbol $\ge$ will be be reserved for
pointwise order. 
} 
for all real-valued functions defined on $\Omega$ 
admits the representation as a probabilistic ranking 
such as
\begin{equation}
\label{rep}
f\ge_*g
\qqtext{if and only if}
f\ge g 
\quad 
m
\text{ almost surely}
\end{equation}
for some reference probability $m$. Given our interest
for finite additivity, justified by the preceding results the 
exact meaning of the expression {\it almost surely}
requires some care. We shall use the expression
$f\ge g$, $m$-a.s. as short for the condition
\begin{equation}
\label{a.s.}
\inf_{t>0}m(f-g<-t)=0.
\end{equation}
When the partial order $\ge_*$ satisfies \eqref{rep}
in the above defined sense, we shall say that $\ge_*$
admits a probabilistic representation, or, if $m$ is known, 
that $\ge_*$ is represented by $m$. We observe that
if $\ge_*$ indeed admits a probabilistic representation 
then it will surely satisfy, among other properties, the
following ones:
\begin{enumerate}[(i).]
\item\label{1>0}
$0\not>_*1$;
\item\label{bounded}
$f\ge_*0$ and $a>0$ imply $f\wedge a\ge_*0$;
\item\label{deterministic}
$f\ge0$ implies $f\ge_*0$;
\item\label{convex}
if $f\ge_*g$ then
$bf+h\ge_*bg+h$ 
for all $b,h:\Omega\to\R$
with $b$ positive and bounded;
\item\label{robust}
if $f+\varepsilon\ge_*0$ for all $\varepsilon>0$ then 
$f\ge_*0$.
\end{enumerate}

If $\ge_*$ is a given partial order, denote by $[f]_*$
the corresponding equivalence class of $f$ 
and write $\Neg_*=\{A\subset\Omega:0\ge_*A\}$.

\begin{theorem}
\label{th kelley order}
A partial order $\ge_*$ defined on $\R^\Omega$
has a probabilistic representation if and only if 
(a)
it satisfies \iref{1>0}--\iref{robust} 
and 
(b)
the following decomposition holds:
\begin{equation}
\label{kelley *}
2^\Omega
	=
\Neg_*\cup\bigcup_{n\in\N}\Bor_n
\qtext{where}
I_*(\Bor_n)
\equiv
\inf_{f\in\Seq[0]{\Bor_n}}\inf_{g\in[f]_*}\sup_{\omega}g(\omega)>0
\quad
n\in\N.
\end{equation}
\end{theorem}

\begin{proof}
Assume \tiref a. By \iref{deterministic} and \iref{convex} 
if $f\ge_*0$ and $t>0$ we have
$
0\ge_*
-f\sset{f<-t}
\ge_*
t\sset{f<-t}
$
and so $\{f<-t\}\in\Neg_*$. Then, $g\in[f]_*$ and $\eta>0$ 
imply
$N_{\eta,g}\equiv\{\abs{f-g}>\eta\}
	\in
\Neg_*$
while
$N\in\Neg_*$ implies $f\set{N^c}\in[f]_*$. Choose $h\in[f]_*$ 
and, since $\Neg_*$ is $\cup$-closed, 
$N_{\eta,h}
	\subset 
N
	\in
\Neg_*$. 
We get
\begin{align}
\label{eq}
\inf_{g\in[f]_*}\sup_{\omega}g(\omega)
	\le
\sup_{\omega\in N^c}f(\omega)
	\le
\eta
+
\sup_{\omega}h(\omega).
\end{align}
In other words, \tiref a implies that
$I_*
=
I_{\Neg_*}
$
(see \eqref{IN}) and, by \iref{1>0} and \iref{convex}, that
$\Neg_*$ is a proper ideal. Therefore, under \tiref a and 
\tiref b Corollary \ref{cor kelley} guarantees that 
$\Neg_*=\Neg(m)$ for some $m\in\Prob(2^\Omega)$: 
$m$ represents $\ge_*$. In fact, $f\ge_*0$ implies 
$\sup_{t>0}m(f<-t)=0$ while $\{f<-t\}\in\Neg_*$ leads 
first to $(f\vee-c)\sset{f<-t}\ge_*0$ for all $c>0$ (by 
\iref{convex}), then to $f\sset{f<-t}\ge_*0$ (by 
\iref{bounded}), hence to
$
f+\varepsilon
\ge_*
(f+\varepsilon)\sset{f\ge-\varepsilon}
\ge_*
0
$
for all $\varepsilon>0$ and, eventually, to $f\ge_*0$, 
(by \iref{robust}).
On the other hand, if $\ge_*$ is the ranking 
induced by some $m\in\Prob(2^\Omega)$,
then $\Neg_*=\Neg(m)$, \tiref a holds and
therefore $I_*=I_{\Neg(m)}$. Let 
$\Bor_n=\{A\subset\Omega:m(A)>1/n\}$ and
$f\in\Seq[0]{\Bor_n}$.
\begin{align*}
1/n
\le
m(f)
=
\inf_{N\in\Neg_*}m(\sset{N^c}f)
\le
\inf_{N\in\Neg_*}\sup_{\omega\in N^c}f(\omega)
=
I_{\Neg_*}(\Bor_n)
=
I_*(\Bor_n)
\end{align*}
so that \tiref b holds as well.
\end{proof}

This last result has a subjective probability interpretation:
a decision maker following a choice criterion that satisfies 
the above conditions \tiref a and \tiref b may be said to
take his decisions on a probabilistic basis. This means that
in principle a probability may be deduced from his behaviour.
We highlight that condition \tiref a would be enough to
imply that $\Neg_*$ is a proper ideal -- and therefore 
that $\Neg_*\subset\Neg(m)$ for some $m\in\Prob(2^\Omega)$
-- but this would not be enough to guarantee that such $m$
is unique (and thus inferable from his decisions).
Uniqueness indeed requires that \eqref{kelley *} is
satisfied, although this appears as a rather difficult
condition to establish in practical problems.

\bibliographystyle{acm}
\bibliography{
../Bib/MathBib,
../Bib/MyBib,
../Bib/EconBib,
../Bib/FinBib}

\end{document}